\documentclass{article}
\usepackage[
backend=biber,
style=numeric, 
sorting=ynt
]{biblatex}
\usepackage[document]{ragged2e}

\usepackage{graphicx} 
\usepackage{amsthm}
\newtheorem{theorem}{Theorem}
\newtheorem{definition}{Definition}
\newtheorem{lemma}[theorem]{Lemma}
\newtheorem{corollary}{Corollary}[theorem]
\newtheorem{remark}{Remark}
\newtheorem{assumption}{Assumption}[theorem]

\usepackage{xcolor}
\usepackage{amsmath}

\usepackage{longtable}
\usepackage{amssymb}
\usepackage{comment}
\usepackage{hyperref}
\usepackage{relsize}

\newcommand{\cK}{{\mathcal K}}

\newcommand{\bbC}{{\mathbb C}}

\usepackage{caption}
\usepackage{blindtext}

\title{Berry--Esseen theorems, LLT, Edgeworth expansions, and large deviations for Locally Stationary Markov Chains}
\author{Brendan Williams, Yeor Hafouta\\
University of Florida and Ben-Gurion University of the Negev}
\date{September 2026}

\begin{document}

\maketitle

\begin{abstract}
Locally stationary Markov chains have attracted a lot of attention in statistics in the past three decades. In this paper
   we prove a variety of limit theorems for partial sums generated by such chains. We first provide explicit formulas for the asymptotic mean and variance (and also higher moments), and obtain optimal convergence rates towards them. We then prove a Berry--Esseen theorem, a local central limit theorem, Edgeworth expansions, and large and moderate deviations principles. Our approach involves a  parametric Perron--Frobenius theorem, which is proved using the theory of complex (Hilbert) projective metrics developed in \cite{Rugh,Dubois}, together with local approximation arguments and ideas in \cite{dolgopyat2023berry}.
\end{abstract}

\section{Introduction}
The central limit theorem (CLT) is one of the most celebrated results in probability theory, stating that for a sequence $(S_n)$, there are $A_n\in\mathbb R$ and $B_n\to\infty$ such that $(S_n-A_n)/B_n$ weakly converges to the standard normal law. Once this is achieved, it is desirable to see when the optimal uniform rates $O(B_n^{-1})$ can be obtained. Other very important related results are the local CLT (LCLT or just LLT), which concerns the asymptotics of probabilities of the form $\mathbb P(S_n\in I)$ for an interval $I$,
and Edgeworth expansions, which are high-order correction terms in the CLT. An additional important statistical property is large deviations, which investigate the rates of decay for probabilities of the form $\mathbb P(S_n/b_n\in\Gamma)$ on a logarithmic scale for appropriate normalizing sequences $b_n\to\infty$ and Borel measurable sets $\Gamma$. 
\vskip0.1cm

In this paper, we consider Markov chains that are locally stationary, which are an important class of triangular arrays of Markov chains. Markov chains are widely used in modeling dependent data. However, a significant limitation of most Markovian models is their reliance on time homogeneity.
As such, locally stationary models have been considered in recent years for processes where stationarity is not reasonable. Roughly speaking, locally stationary Markov chains are arrays of Markov chains that can be approximated by a family of stationary Markov chains. 
Locally stationary models were first introduced in \cite{dahlhaus1997fitting}, and have been of interest in recent years. See also \cite{dahlhausOver} for an overview.
Applications include locally stationary analogues of likelihood approximation, time-homogeneous, autoregressive processes and finite-state Markov chains: see, for example, \cite{dahlhausLikelihood}, \cite{dahlhaus2006statistical}, \cite{rao2006some} and \cite{vogt2012nonparametric}. The CLT for certain classes of locally stationary chains was studied in \cite{CLT} and in a more general framework in \cite{truquet2019local}. See, for instance, also \cite{CLT1}, \cite{CLT2} and \cite{CLT3}. 
Compared with the situation of an inhomogeneous chain \cite{dolgopyat2020local,dolgopyat2023berry, DH, HW, MagdaLLT1, PelLLT} (i.e., a sequence), the issue here is that we only have a triangular array of Markov operators, and so the arguments in the above results for non-stationary chains depend on the parameter $n$ of the array. In particular, certain constants that appear in the arguments might grow in $n$. Moreover, it is desirable to obtain results under conditions on each of the stationary chains and not on the triangular array itself. 
\vskip0.1cm

Despite being studied in the literature by many authors (as discussed above), limit theorems beyond the CLT for locally stationary chains have not been obtained. 
In this paper, we will close this gap.
Let $\{X_{k,n}: 1\leq k\leq n\}$ be the locally stationary array.
We consider partial sums of the form $S_n=\sum_{j=1}^ng(k/n,X_{k,n})$ for bounded functions $g(u,x)$ which are Lipschitz continuous in $u$, uniformly in $x$. 
\vskip0.1cm
We first study the behavior of the mean and the variance of $S_n$. We compute explicitly the asymptotic mean $b$ and variance $a^2>0$ by means of a process $Z_j=(U,V_j)$ whose conditional distributions with respect to $U=u$ form the locally stationary Markov chain corresponding to the parameter $u$. Here, $U$ is uniformly distributed on $[0,1]$.  We also show that $|\mathbb E [S_n]-bn|$ and $|\text{Var}(S_n)-a^2n|$ are bounded. We also obtain similar estimates for high-order cumulants and moments of $S_n$.
Then we obtain optimal CLT rates for $(S_n-bn)/(a\sqrt n)$.  After that, under certain irreducibility conditions on the functions $g(u,\cdot)$, we obtain an LLT and first-order correction terms in the CLT. When the functions $g(u,x)$ are H\"older continuous in both $u$ and $x$ and the stationary kernels have positive bounded densities, we prove high-order correction terms (i.e., Edgeworth expansions). When $g$ is Lipschitz continuous, we obtain expansions of all orders. We also prove a local large deviations principle and global moderate deviations principle with optimal scale.
\vskip0.1cm

As noted above, the main ideas in the proof rely on local approximation by stationary chains.
A key tool in the proof is a parametric Perron--Frobenius theorem; see Theorem \ref{RPF theorem single operator}. This is proved by using contraction properties of complex projective metrics developed by Rugh \cite{Rugh} and Dubois \cite{Dubois} together with ideas in \cite{HK}. The reason we use complex projective metrics and not more classical tools like inverse function theorems is that the former provides uniform estimates in the parameter defining the local stationarity, as well as bypasses the fact that the local stationary distribution depends on the parameter. 
In addition, 
a lot of effort is dedicated to controlling the behavior of the eigenvalues with respect to the parameter defining the local stationarity.  
\vskip0.1cm

Apart from the Perron--Frobenius theorem discussed above, 
our approach for the variance and mean growth is partially based on ideas in \cite{HK,koralov2025local}. Our approach to the Berry--Esseen theorem is partially based on ideas in \cite{HK, dolgopyat2023berry,koralov2025local}. Our approach for the local limit theorem and the first-order expansions is partially based on ideas in \cite{koralov2025local}. Our approach to the Edgeworth expansion is based on ideas in \cite{dolgopyat2023berry} and \cite{PelLLT}. Our proof of the large and moderate deviations follows ideas in \cite{DH1}.

\section{Preliminaries}
Let $(E,d)$ be a metric space and let $\{Q_u: u\in [0,1]\}$ be a family of transition probabilities  on $(E, \mathcal B(E))$, where $\mathcal B(E)$ is the Borel $\sigma$-algebra on $E$. That is, $Q_u(x,\cdot)$ is a probability measure on $(E, \mathcal B(E))$  for all $x\in E$. 
We consider triangular arrays of random variables $\{X_{n,j}: 1\leq j\leq n \} $ such that for all $n\in \mathbb Z^+$,  the sequence $(X_{n,j})_{1\leq j\leq n}$ is an inhomogeneous Markov chain and for all $x\in E$ and $A\in\mathcal B(E)$,
\begin{equation}\label{LS}
\mathbb P (X_{n,k} \in A \, |\, X_{n,k-1} = x) = Q_{k/n}(x,A),\,\,\, 2\leq k \leq n.    
\end{equation}

We now fix our notation. For all integers $j\geq 1$, we denote by $\mathcal{P}(E^j)$ the set of probability measures on $(E^j, \mathcal{B}(E^j))$. We define $\mathfrak B$ to be the Banach space of bounded $\mathcal{B}(E)$-measurable complex functions on $E$ equipped with the supremum norm $\|f\|_\infty = \sup_{x\in E} |f(x)|$. We define $\mathcal L_{\mathfrak B}$ to be the space of all bounded linear operators from $\mathfrak B$ to itself. We denote by $\mathfrak B'$ the dual space of $\mathfrak B$.
For $\varphi\in \mathfrak B'$ and $\upsilon \in \mathfrak B$, we define $\upsilon \otimes \varphi \in \mathcal L_{\mathfrak B}$ by $(\upsilon \otimes \varphi)f = \langle \varphi, f \rangle \upsilon$ for all $f\in \mathfrak B$. We let $r(\cdot)$ denote the spectral radius of an operator with respect to the norm on $\mathcal L_{\mathfrak B}$.

\begin{definition}
    For $u\in [0,1]$, let $\pi_u$ denote the invariant probability measure of $Q_u$, and let $\pi_{u,j}$ denote the $j$-dimensional marginal distribution of the stationary Markov chain with transition kernel $Q_u$. Let $\pi_{k,j}^{(n)}$ denote the distribution of $(X_{n,k}, \cdots, X_{n, k+j-1})$. We say that $\{X_{n,k}: 1\leq k\leq n\}$ is locally stationary if, for every $j\geq 1$, $u\mapsto \pi_{u,j}$ is continuous in total variation and $$\lim_{n\to\infty}\sup_{1\leq k\leq n-j+1}\| \pi_{k,j}^{(n)} - \pi_{k/n,j}  \|_{\operatorname{TV}} = 0.$$
    We assume that $X_{n,0}$ has probability distribution $\pi_0$.
\end{definition}

Note that determining the invariant probability measure can be challenging in practice. See, for example, \cite{truquet2019local}  for inference of the invariant measures for finite-state Markov chains.

\begin{definition}
    For $u\in [0,1]$, define the operator $\Pi_u \in \mathcal L_{\mathfrak B}$  by $\Pi_u:= \mathbf{1} \otimes \pi_u$. Equivalently, for each $f\in \mathfrak B$,
    $$\Pi_u f = \pi_u(f)\mathbf{1}.$$
\end{definition}

We consider statistics of the form $$S_n = \sum_{k=1}^n g(k/n, X_{n,k})$$ for a bounded measurable function $g:[0,1]\times E \to \mathbb R$.

We denote by $\|\cdot \|_{\text{TV}}$ the total variation norm and by $\|\cdot\|_\infty$ the operator norm induced by the supremum norm $\|\cdot\|_{\infty}$. In particular, for each $T\in \mathcal{L}_{\mathfrak B}$,  $$\|T\|_\infty := \sup_{\|f\|_{\infty} \leq 1} \|Tf\|_{\infty}. $$
Furthermore, for Markov operator $Q_u$, $Q_v$, 
$$\|Q_u - Q_v\|_\infty = \sup_{x\in E}\| \delta_x Q_u - \delta_x Q_v    \|_{\text{TV}},$$ where $$\| \mu-\nu \|_{\text{TV}} = \sup_{\|f\|_{{\infty}} \leq 1} \left|\int f d\mu - \int fd\nu\right| = 2\sup_{A\in\mathcal B(E)} |\mu(A) - \nu(A)|$$
and  $\delta_x$ denotes the Dirac measure concentrated at $x$.

For $\varphi\in\mathfrak B'$, we use the dual norm
$$ \|\varphi\|_{\mathfrak B'} := \sup_{\|f\|_{\infty}\leq 1} |\langle\varphi,f\rangle|. $$

Our results will use the following.
\begin{definition}
    For $f\in \mathfrak B$, $u\in[0,1]$ and $z\in \mathbb C$, the twisted Markov operator $Q_{u,z}\in\mathcal L_{\mathfrak B}$ is defined by $$Q_{u,z}f(x) = \int e^{zg(u,y)} f(y) Q_u(x, dy).$$
\end{definition}

Our methods are based on the following observation.
From the definition of $S_n$, we immediately get the following result.
\begin{lemma}
    For each $f\in \mathfrak B$, $z\in\mathbb C$, and every Borel probability measure $\mu$ on $E$, 
    $$\mathbb E_{\mu}[e^{zS_n}f(X_{n,n})] = \left\langle  \mu,  Q_{1/n,z}Q_{2/n,z}\cdots Q_{n/n,z}f     \right\rangle.$$
\end{lemma}

\section{Main results}

\subsection{Standing assumptions}
In this section, suppose that $\{X_{n,k}: 1\leq k\leq n\}$ satisfies \eqref{LS} and consider the following cascade of assumptions. 
\begin{assumption}\label{MainAssm}
\,
\vskip0.1cm

    \begin{enumerate}
    \item[(i)]   There exists a positive real number $L$ such that for all $(u,v)\in [0,1]^2$, $$\|Q_u - Q_v\|_{\infty} = \sup_{x\in E}\|\delta_x Q_u - \delta_xQ_v\|_{\operatorname{TV}}\leq L|u-v|.$$ \label{assumption 1}
\item[(ii)] There exists an integer $m\geq 1$ and $r\in (0,1)$ such that $$\sup_{u\in[0,1]}\sup_{x,y\in E}\|\delta_x Q_u^m - \delta_yQ_u^m\|_{\operatorname{TV}}\leq r.$$ \label{assumption 2}
  \item[(iii)] There exists $K<\infty$ such that $\sup_{u\in[0,1]}\|g(u,\cdot)\|_\infty\leq K.$ 
    \item[(iv)] The observable is Lipschitz in $u$: There exists $C_g \geq 0$ such that
$$
|g(u,x)-g(v,x)|\le C_g|u-v|.
$$
\item[(v)] A non-arithmetic condition holds: For every $u\in [0,1]$ and $t\neq 0$, $r(Q_{u,it})<1.$
\end{enumerate}
\end{assumption}

 The non-arithmetic condition is not necessary for the Perron--Frobenius results and the optimal CLT rates, but will be for the local limit theorem and the Edgeworth expansions.

 Truquet \cite{truquet2019local} showed that Assumptions \ref{MainAssm} (i)-(ii) imply the following result.
\begin{theorem}[Theorem 1 in \cite{truquet2019local}]\label{Tru}
    Make Assumptions \ref{MainAssm} (i)-(ii). Then, for all $u\in[0,1]$, the Markov kernel $Q_u$ has a unique invariant probability measure $\pi_u$. The triangular array of Markov chains $\{X_{n,k}: n\in \mathbb Z^+, k\leq n\}$ is locally stationary for the total variation distance. Moreover, for all integers $j\geq 1$, there exists a positive real number $C_j$, not dependent on $k,n,u$ and such that $$\| \pi_{k,j}^{(n)} - \pi_{u,j}\|_{\operatorname{TV}}\leq C_j \left[|u-k/n|+1/n \right].$$
\end{theorem}
Thus, the chains considered in this paper are locally stationary.
\subsection{Limit theorems}

\begin{theorem}[Asymptotic variance and mean]\label{Var}
Make Assumptions \ref{MainAssm} (i)-(iv).
\vskip0.1cm

(1) For $u\in[0,1]$ set 
$$
b_u=\int_E g(u,x)d\pi_u(x),
$$
and let
$b=\int_{0}^1b_u\,du=\int_{0}^1\int_E g(u,x)d\pi_u(x)du$. Then
$$
\lim_{n\to\infty}\frac1n\mathbb E[S_n]=b.
$$
Additionally,
$$
\sup_n \left|\mathbb E[S_n]-nb\right|<\infty.
$$
\vskip0.1cm
(2) Let  $(Y_{j,u})_{j=0}^\infty$ be the stationary Markov chain with transition operator $Q_u$, and for $u\in[0,1]$, set 
$$
a_u^2=\lim_{n\to\infty}\frac1n\operatorname{Var}\left(\sum_{j=1}^n g(u,Y_{j,u})\right)
$$
$$
=\int_E \tilde g^2(u,x)d\pi_u(x)+2\sum_{j=1}^\infty\int_E\tilde g(u,x)\left(Q_u^j(\tilde g(u,\cdot))\right)(x)d\pi_u(x),
$$
where $\tilde g(u,x)=g(u,x)-\int g(u,y)d\pi_u(y)$.
Let
$a^2=\int_{0}^1a_u^2\,du$. Then
$$
\lim_{n\to\infty}\frac1n\operatorname{Var}(S_n)=a^2.
$$
Moreover, 
$$
\sup_n\left|\operatorname{Var}(S_n)-na^2\right|<\infty
$$
and 
$$
\sup_n\left|\mathbb E[(S_n-bn)^2]-na^2\right|<\infty.
$$
Finally, when Assumption \ref{MainAssm} (v) holds (even only for small $t\not=0$), we have $a^2>0$ (in fact, $a_u^2>0$ for all $u$).
\end{theorem}
We refer to Theorem \ref{HighThm} in Section \ref{Rem1} for computations of higher-order asymptotic moments of $S_n$, which we decided not to present in this section in order not to overload it. 
\begin{remark}
Our analysis shows that $a_u^2$ is continuous in $u$, and so $a^2>0$ even when $r(Q_{u,it})<1$ for all sufficiently small real $t\not=0$ and $u$ belonging to some subinterval of $[0,1]$.    
\end{remark}
\begin{remark}
Note that $b=\mathbb E[g(Z)]$, where $Z=(U,Y)$, where  $U$ is uniformly distributed on $[0,1]$, $Y$ is $E$-valued and the conditional law of $Y$ given $U=u$ is $\pi_u$. To describe $a$, denote $\tilde g(u,x)=g(u,x)-\int_E g(u,y)d\pi_u(y)$. Then 
$$
a^2=\lim_{n\to\infty}\frac1n\mathbb E\left[\left(\sum_{j=1}^n \tilde g(Z_j)\right)^2\right],
$$
where $Z_j=(U,V_j)$, $U$ is uniformly distributed on $[0,1]$, $V_j$ are $E$-valued and the conditional law of $(V_j)_{j=1}^\infty$ given $U=u$ coincides with the law of $(Y_{j,u})_{j=1}^\infty$. Thus, $a^2$ is the asymptotic variance of the partial sum $\sum_{j=1}^n \tilde g(Z_j)$.
 \end{remark}
 Our next result concerns optimal CLT rates.
\begin{theorem}[A Berry--Esseen theorem]\label{BE}
Under Assumptions \ref{MainAssm} (i)-(iv) and when $a>0$, 
there exists a constant $C>0$ such that for all $n$, 
$$
\sup_{t\in\mathbb R} \left|\mathbb P(S_n-bn\leq a\sqrt n t)-\Phi(t)\right|\leq C/\sqrt n,
$$
where $\Phi(t)=\frac1{\sqrt{2\pi}}\int_{-\infty}^te^{-\frac12 x^2}dx$.
\end{theorem}
The following result concerns the LLT.
\begin{theorem}[A local limit theorem]\label{LLT}
Under Assumptions \ref{MainAssm} (i)-(v), 
for every continuous function $G:\mathbb R\to\mathbb R$ with compact support or an indicator of a bounded interval, we have
$$
\lim_{n\to\infty}\sup_{v\in\mathbb R}\left|a\sqrt {2\pi n}\mathbb E[G(S_n-v)]-\left(\int G(x) dx\right)e^{-\frac{(v-bn)^2}{2na^2}}\right|=0.
$$
\end{theorem}
Our next result is about first-order correction terms in the CLT.
\begin{theorem}[First-order Edgeworth expansions]\label{Edge0}
Under Assumptions \ref{MainAssm} (i)-(v),
there exists a polynomial $P_1$ of degree $3$ such that  
$$
\sup_{t\in\mathbb R}\left|\mathbb P(S_n-bn\leq at\sqrt n)-\Phi(t)-n^{-1/2}P_{1}(t)\varphi(t)\right|=o(n^{-1/2}),
$$
where $\varphi(t)=\frac{1}{\sqrt{2\pi}}e^{-\frac12t^2}$.
\end{theorem}
In fact, we can take $P_1(t)=\frac16(t^3-t)\nu_3$, where $\nu_3=\lim_{n\to\infty}\frac1n\mathbb E[(S_n-bn)^3]$ (the limit $\nu_3$ exists by Theorem \ref{HighThm}).
The following two results concern high-order correction terms in the CLT.
\begin{theorem}[High-order Edgeworth expansions I]\label{Edge1}
Let  Assumptions \ref{MainAssm} (i)-(v)  be in force. There are polynomials $P_{j}$ with the following properties.
Let $d>1$ be an integer. Suppose also that
\vskip0.1cm 
(1) there is $a_n\in(0,1]$ such that for all $k\leq n$,
 $$
Q_{k/n}(x,\Gamma)\geq a_n\mathbb P(X_{k,n}\in\Gamma)
 $$
 for all $x$ and all measurable sets $\Gamma$. 
 \vskip0.1cm
 (2) for all $A,B>0$,
 $$
\int_{A\leq |t|\leq Bn^{(d-1)/2}}\exp\left(-\frac{a_n^4}{16}\sum_{k=1}^n\alpha_{k/n}(t)\right)/|t|\,dt=o(n^{-d/2}),
 $$   
 where 
 $$
 \alpha_u(t)=1-\left|\int e^{itg(u,x)}d\pi_u(x)\right|^2.
 $$
\vskip0.1cm
Then 
$$
\sup_{t\in\mathbb R}\left|\mathbb P(S_n-bn\leq at\sqrt n)-\Phi(t)-\sum_{j=1}^{d}n^{-j/2}P_{j}(t)\varphi(t)\right|=o(n^{-d/2}).
$$
\end{theorem}
\begin{remark}
 Suppose $\inf a_n>0$. Let $Y_u$ be an $E$-valued random variable distributed according to $\pi_u$. Then the conditions of Theorem \ref{Edge1} hold when the random variables $H_u=g(u,Y_u)$
 are Diophantine (see \cite{Bre}) with an appropriate order and uniform constants, that is, when  $\sup_u|\mathbb E[e^{itH_u}]|\leq 1-C/|t|^\ell, |t|>0$ for some $C>0$ and an appropriate $\ell=\ell_d$. In particular, they hold for all $d$ if the uniform Cramer condition 
 $$
 \limsup_{|t|\to\infty}\sup_{u\in[0,1]}\left|\mathbb E[e^{itH_u}]\right|<1
 $$ 
 holds.
\end{remark}
\begin{remark}
By Lemma \ref{Lipschitz invariant measure lemma}, the map $u\to\pi_u$ is Lipschitz continuous in the total variation. However, in general the Lipschitz constant of the function $u\to \alpha_u(t)$ is of order $|t|$, and so for large values of $|t|$ it is not possible to replace the sum   $\sum_{k=1}^n\alpha_{k/n}(t)$ with $n\int \alpha_u(t)du$ (since the error term after taking the exponents is of order $e^{c|t|}$ for some $c>0$).
\end{remark}

\begin{theorem}[High-order Edgeworth expansions II]\label{Edge2}
Let  Assumptions  \ref{MainAssm} (i)-(v) be in force. There are polynomials $P_{j}$ with the following properties. If
\vskip0.1cm
(1) the state space  $E$ of the chains  is a  compact  Riemannian  manifold, and there are positive functions $p_u: E\times E\to\mathbb R, u\in[0,1]$, uniformly bounded and bounded away from $0$, such that
 for all $u\in[0,1]$, $x\in E$ and $\Gamma\in\mathcal B(E)$,
$$
Q_u(x,\Gamma)=\int_{\Gamma}p_u(x,y)dy,
$$
where $dy$ denotes the normalized volume measure.
\vskip0.1cm
(2) $u\to p_u$ is  Lipschitz continuous with respect to the supremum norm.  
\vskip0.1cm
(3) the functions  $g_u(x)=g(u,x)$ are uniformly H\"older continuous with some exponent $\alpha\in(0,1]$. 
\vskip0.1cm
Then  for every positive integer $d<\frac{1+\alpha}{1-\alpha}$, we have  
$$
\sup_{t\in\mathbb R}\left|\mathbb P(S_n-bn\leq at\sqrt n)-\Phi(t)-\sum_{j=1}^{d}n^{-j/2}P_{j}(t)\varphi(t)\right|=o(n^{-d/2}),
$$
where for $\alpha=1$ we set $\frac{1+\alpha}{1-\alpha}=\infty$.
\end{theorem}
\begin{remark}
The polynomials in Theorems \ref{Edge0}, \ref{Edge1} and \ref{Edge2} coincide. We can actually provide a closed formula for the coefficients of the polynomials $P_j$. In order not to overload the introduction, we will only briefly discuss it. First,
similarly to Theorem \ref{Edge0}, the polynomial $P_j$ in Theorem \ref{Edge1} and \ref{Edge2} has degree $j+2$ and its coefficients can be expressed by an algebraic combination of the asymptotic cumulants (and so moments) of $S_n$ up to order $j+2$, which exist by Theorem \ref{HighThm}. Theorem \ref{HighThm} also shows that the latter $k$-th asymptotic cumulant equals 
$$
\int_0^1\gamma_{k}(u)du,
$$
where $\gamma_k(u)$ is the $k$-th asymptotic cumulants of the sum 
$S_{n,u}=\sum_{j=1}^n g(u,Y_{j,u})$
and $(Y_{j,u})_{j=0}^\infty$ is the stationary Markov chain with transition operator $Q_u$.
\end{remark}
Our next result is a local large deviations principle (LDP).
\begin{theorem}[Local LDP]\label{LDP}
Let  Assumptions \ref{MainAssm} (i)-(iv)  be in force and suppose $a>0$.
There exists $\epsilon_0>0$ and a positive convex function $c:(0,\epsilon_0)\to\mathbb R$ such that for all $0<\epsilon<\epsilon_0$,
$$
\lim_{n\to\infty}\frac1n\ln\mathbb P(S_n-bn\geq \epsilon n)=-c(\epsilon).
$$
\end{theorem}
We are also able to prove a global moderate deviations principle (MDP) with optimal scale.
\begin{theorem}[MDP]\label{MDP}
Let  Assumptions \ref{MainAssm} (i)-(iv) be in force and suppose $a>0$.
For every sequence $c_n$ such that $\lim_{n\to\infty}\frac{c_n}{\sqrt n}=\infty$ and $\lim_{n\to\infty}\frac{c_n}{n}=0$, for all open sets $\Gamma\subset\mathbb R$ we have 
$$
-\inf_{x\in\Gamma}\left(\frac{x^2}{2a^2}\right)\leq \liminf_{n\to\infty}\frac{1}{c_n^2/n}\ln\mathbb P((S_n-bn)/c_n\in\Gamma),
$$
and for all closed sets $\Gamma\subset\mathbb R$ we have 
$$
\limsup_{n\to\infty}\frac{1}{c_n^2/n}\ln\mathbb P((S_n-bn)/c_n\in\Gamma)\leq -\inf_{x\in\Gamma}\left(\frac{x^2}{2a^2}\right).
$$
\end{theorem}

\section{A parametric Perron--Frobenius Theorem}
Before proving a Perron--Frobenius theorem, we develop some lemmas that ensure a certain structure under perturbations, under our assumptions. 

\subsection{Perturbation Lemmas}

Truquet \cite{truquet2019local} proved that Assumptions 1 and 2 imply geometric $\phi$-mixing (Proposition 1, \cite{truquet2019local}). For $1\leq i \leq j \leq n$, we set $$\mathcal{F}_{i,j}^{(n)} = \sigma(X_{n, \ell}  : i\leq \ell \leq j),$$
 and for $0\leq j \leq n-1$, $$\phi_n(j) = \max_{1\leq i \leq n-j} \sup \left\{|\mathbb P(B \, | \, A) - \mathbb P(B)|: B\in \mathcal F_{i+j, n}^{(n)}, A\in \mathcal{F}_{1,i}^{(n)}, \mathbb P(A) \neq 0  \right\}.$$ Finally, we set $\phi(j) = \sup_{n>j} \phi_n(j)$.

\begin{lemma}
    Make Assumptions  \ref{MainAssm} (i)-(ii). Then, there exists $C>0$ and $\rho\in (0,1)$, only depending on $m$, $L$, and $r$ such that $$\phi(j)\le C\rho^j.$$
\end{lemma}

The following lemma demonstrates the existence of a uniform spectral gap.

\begin{lemma} \label{Spectral Decomposition Lemma}
For every $u\in[0,1]$, we have the following uniform spectral-gap decomposition.  $$Q_u=\Pi_u+N_u, \qquad N_u:=Q_u-\Pi_u,$$ 
   where 
   $$\Pi_u^2=\Pi_u,\qquad\Pi_uN_u=N_u\Pi_u=0,$$ and there exist constants $C > 0$ and $\rho\in (0,1)$ such that for each $n\in \mathbb Z^+$, $$\sup_{u\in [0,1]}\|N_u^n\|_\infty=\|Q_u^n-\Pi_u\|_\infty \leq C\rho^n.$$
 Furthermore, $$\sup_{u\in[0,1]}r(N_u)<1.$$
That is, $\{Q_u:u\in[0,1]\}$ has a uniform spectral gap on $\mathfrak B$.
\end{lemma}

\begin{proof}
Fix  $u\in [0,1]$. By Assumption \ref{MainAssm} (ii), we have
$$\left\|\delta_xQ_u^m-\pi_u\right\|_{\text{TV}}=
\left\|\delta_xQ_u^m-\int_E\delta_yQ_u^m\pi_u(dy)\right\|_{\text{TV}}
$$
$$
\leq
\int_E\left\|\delta_xQ_u^m-\delta_yQ_u^m\right\|_{\text{TV}} \pi_u(dy)
\leq r.$$
Since $\Pi_uQ_u=Q_u\Pi_u=\Pi_u$, we have
$$N_u^m=(Q_u-\Pi_u)^m=Q_u^m-\Pi_u.$$ 
Thus, for $f\in\mathfrak B$, $$|(Q_u^m - \Pi_u)f(x)| = \left|\int f d(\delta_x Q_u^m - \pi_u)\right|\leq r\|f\|_\infty.$$
Thus, we have
$$\|N_u^m\|_\infty\leq r<1.$$
Because the estimate is uniform in $u$, and by decomposing each $n\in \mathbb Z^+$ as $n = mq+\ell$, it follows that there exist $C>0$ and
$\rho=r^{1/m}\in(0,1)$ which are independent of $n$ such that
$$
\sup_{u\in[0,1]}\|N_u^n\|_\infty
\leq C\rho^n.
$$
\end{proof}

Under the assumption that the function $g$ is bounded, for $z\in\mathbb C$,
denote 
$$
Q_{u,z}h(x)=\int e^{zg(u,y)}h(y)Q_u(x,dy).
$$

\begin{lemma}\label{Derivatives of twisted transfer operator lemma}
    Under Assumptions \ref{MainAssm} (i)-(iv),
    for $j\in \mathbb Z^+$ and $z\in\mathbb C$,
    $$\sup_{u\in[0,1]}\left\|\frac{\partial^j}{\partial z^j}Q_{u,z}\right\|_\infty \le K^je^{K|\Re(z)|},$$ 
    where $\Re (z)$ denotes the real part of $z$ and $K=\sup_{u\in[0,1],x\in E}|g(u,x)|$.
\end{lemma}

\begin{proof}
Let $f\in\mathfrak B$, $x\in E$ and $j\in \mathbb Z^+$. By differentiating under the integral of 
$$
Q_{u,z}f(x) = \int e^{z g(u,y)} f(y) Q_u(x, dy),
$$
we obtain that
$$
\frac{\partial^j}{\partial z^j}Q_{u,z}f(x) = \int (g(u,y))^j e^{zg(u,y)} f(y) Q_{u}(x,dy).
$$ 
Set $K=\sup|g|$.
Taking the supremum over $x\in E$ gives 
$$
\left\|\frac{\partial^j}{\partial z^j}Q_{u,z} f\right\|_{\infty}\leq K^j e^{|\Re(z)|K}\|f\|_\infty.
$$
The proof is completed by taking the supremum over $f$ with $\|f\|_\infty\leq 1$.
\end{proof}



\begin{lemma}\label{Twisted operator lemma}
Under  Assumptions \ref{MainAssm} (i)-(iv), for every $z\in\mathbb C$, 
$$\sup_{u\in[0,1]}\|Q_{u,z}-Q_u\|_\infty\leq K|z| e^{K|\Re(z)|},$$
\end{lemma}
where $K=\sup_{u\in[0,1],x\in E}|g(u,x)|$.
\begin{proof}
 For $f\in \mathfrak B$ and $x\in E$, we have
\begin{align*}
\left|(Q_{u,z}-Q_u)f(x)\right| &= \left| \int_E \left(e^{z g(u,y)}-1\right)f(y)\,Q_u(x,dy) \right|\\
&\leq \sup_{y\in E}\left|e^{zg(u,y)}-1\right| \|f\|_\infty\\
&\leq |z|\|g(u,\cdot)\|_\infty e^{|\Re(z)|\|g(u,\cdot)\|_\infty }\|f\|_\infty
\leq K|z| e^{K|\Re(z)|}\|f\|_\infty.
\end{align*}
Taking the supremum over $x\in E$ and $f$ with $\|f\|_\infty\leq 1$ completes the proof.
\end{proof}

\begin{lemma}\label{Lipschitz invariant measure lemma}
   Under  Assumptions \ref{MainAssm} (i)-(iv), for $u,v\in [0,1]$, we have $$\|\pi_u-\pi_v\|_{\operatorname{TV}}
\leq \frac{mL}{1-r}|u-v|, \text{ and} \quad \|\Pi_u-\Pi_v\|_\infty  \leq C|u-v|.$$
\end{lemma}

\begin{proof}
Assumption \ref{MainAssm} (ii) gives $\sup_{u,x,y}\|\delta_x Q_u^m - \delta_yQ_u^m\|_{\text{TV}}\leq r.$ Let $P(x,dy)$ be a family of transition probabilities on $(E,\mathcal B(E))$.
Recall the Dobrushin contraction inequality $\| \mu P - \nu P\|_{\text{TV}} \leq \delta(P)\|\mu-\nu\|_{\text{TV}}$ for probability measures $\mu,\nu$, where  the Dobrushin contraction coefficient $\delta(P)$ is given by 
$$
\delta(P)= \sup_{x,y\in E}\|\delta_x P - \delta_y P\|_{\text{TV}}.
$$ 
Together, these imply that $$\|\mu Q_u^m - \nu Q_u^m\|_{\text{TV}} \leq r \|\mu-\nu\|_{\text{TV}}. $$ Since
$\pi_uQ_u^m=\pi_u$ and $\pi_vQ_v^m=\pi_v$, we obtain
\begin{align*}
\|\pi_u-\pi_v\|_{\text{TV}} &\leq \|\pi_uQ_u^m-\pi_vQ_u^m\|_{\text{TV}} + \|\pi_vQ_u^m-\pi_vQ_v^m\|_{\text{TV}}\\
&\leq r\|\pi_u-\pi_v\|_{\text{TV}} + \|\pi_v(Q_u^m-Q_v^m)\|_{\text{TV}}.
\end{align*} We now bound the second term. By the telescoping expansion
$$Q_u^m-Q_v^m = \sum_{\ell=0}^{m-1} Q_u^\ell(Q_u-Q_v)Q_v^{m-1-\ell},$$
we have
 $$\|\pi_v(Q_u^m-Q_v^m)\|_{\text{TV}}
\leq mL|u-v|.$$ Thus, $$\|\pi_u-\pi_v\|_{\text{TV}} \leq \frac{mL}{1-r}|u-v|.$$

Furthermore, for $f\in \mathfrak B$, 
$$(\Pi_u-\Pi_v)f  =  \big(\pi_u(f)-\pi_v(f)\big)\mathbf{1}.$$
Thus, by taking the supremum over $\|f\|_\infty < 1$,
$$\|\Pi_u-\Pi_v\|_\infty = \|\pi_u-\pi_v\|_{\text{TV}} \leq \frac{mL}{1-r}|u-v|.$$ Thus, the proof is complete.

\end{proof}



\subsection{A parametric Perron--Frobenius Theorem}

\begin{theorem}[Parametric Perron--Frobenius Theorem]\label{RPF theorem single operator}
Under  Assumptions \ref{MainAssm} (i)-(iv), the following Perron--Frobenius theorem holds. 

There exist $\delta,C>0$, and $r\in(0,1)$ such that, for $|z|< \delta$, and $u\in [0,1]$, $Q_{u,z}$ has the following decomposition:
\begin{align}
    Q_{u,z} = \lambda(u,z) \upsilon(u,z) \otimes \varphi(u,z) + N_{u,z},\label{rpf decomposition}
\end{align}

where
\begin{enumerate}
    \item The leading eigenvalue $\lambda(u,z)$ is simple and isolated.
    \item $\upsilon(u,z)\in\mathfrak B$ and functionals
    $\varphi(u,z)\in\mathfrak B'$ satisfy $$Q_{u,z}\upsilon(u,z)=\lambda(u,z)\upsilon(u,z), \text{ and } \qquad \varphi(u,z)Q_{u,z}
=\lambda(u,z)\varphi(u,z),$$ 
    with the normalization $\langle\varphi(u,z),\upsilon(u,z)\rangle=1.$ 
    \item $\mathfrak{B} = F_{u,z} \oplus H_{u,z}$  for each $u\in [0,1]$, where
    $$F_{u,z} = \operatorname{span}\{ \upsilon(u,z)  \}, \text{ and } \qquad  H_{u,z} = \{ h: \langle  \varphi(u,z),h  \rangle = 0  \}.$$
    \item Denoting $\Pi_{u,z}:= \upsilon(u,z)\otimes \varphi(u,z)$ and $\Pi_{H_{u,z}}:= I-\Pi_{u,z}$, we have 
    $$N_{u,z} = Q_{u,z}\Pi_{H_{u,z}}, \text{ and } \qquad \sup_{u\in [0,1]}\sup_{|z|< \delta}r(N_{u,z})/|\lambda(u,z)|<r$$

    \item
$\sup_{u\in[0,1]}\sup_{|z|< \delta}\sup_{n\in \mathbb Z^+}\left\|\frac{\partial^2}{\partial z^2}N_{u,z}^n\right\|_\infty\leq C.$ 
    \item The spectral data are $C^1$ (analytic) in the  parameter $z$. In particular,   
    $$ z\mapsto\lambda(u,z),   \qquad  z\mapsto\upsilon(u,z), \qquad z\mapsto\varphi(u,z), \qquad z\mapsto N_{u,z}$$
     are analytic for every $u\in[0,1]$, and for each $j=1,2,3$, 
$$\sup_{u\in[0,1]}  \sup_{|z|< \delta} |\partial_z^j\lambda(u,z)|  \leq C,$$
    $$\sup_{u\in[0,1]} \sup_{|z|< \delta} \|\partial_z^j\upsilon(u,z)\|_\infty \leq C, $$
   $$\sup_{u\in[0,1]} \sup_{|z|< \delta} \|\partial_z^j\varphi(u,z)\|_{\mathfrak B'} \leq C,$$
    and
    $$\sup_{u\in[0,1]} \sup_{|z|< \delta} \|\partial_z^jN_{u,z}\|_\infty \leq C. $$
 \item The spectral data are Lipschitz in the parameter $u$. That is,
    $$ |\lambda(u,z)-\lambda(v,z)| \leq C|u-v|, $$
    $$ \|\upsilon(u,z)-\upsilon(v,z)\|_\infty \leq C|u-v|,$$
    $$\|\varphi(u,z)-\varphi(v,z)\|_{\mathfrak B'} \leq C|u-v|,$$
    and
    $$ \|N_{u,z}-N_{v,z}\|_\infty\leq C|u-v|.$$
\end{enumerate}

\end{theorem}
\begin{proof}
Recall the notation
$H_{u} = \{ h \in \mathfrak B: \langle  \pi_u,h  \rangle = 0  \}$.
 By  Lemma \ref{Spectral Decomposition Lemma}, because $1$ is not in the spectrum of $Q_u|_{H_u}$, and  by the uniform spectral gap of $Q_u$,  there are constants $A>0$ and $\rho\in(0,1)$ such that for all $n$, we have 
\begin{equation}\label{Uni}
 \sup_{u\in[0,1]}\|Q_u^n-\pi_u\|_\infty\leq A\rho^n.   
\end{equation}
Fix some $C>2$ and consider the cone 
$$
\mathcal K_C=\{h\geq 0: h(y)\leq Ch(x),\,\,\forall x,y\in E\}.
$$
Let $h\in\mathcal K_C$ and suppose $\|h\|_\infty=1$. Then for all $u$ we have $\pi_u(h)\geq C^{-1}$. Therefore by \eqref{Uni} there exists $N_C$ such that for all $n\geq N_C$ and all $u\in[0,1]$, we have
$$
Q_u^n\mathcal K_C\subset\mathcal K_{C/2}.
$$
Now, recall that there is $0<d_C<\infty$ such that the projective diameter of $\mathcal K_{C/2}$ inside $\mathcal K_{C}$ does not exceed $d_C$, see \cite[Lemma 6.5.1]{HK}. We also refer to \cite[Appendix A]{HK} for the definition of the projective metric associated with a cone $\mathcal K$.

Let $\Gamma_C$ be the set of all linear functionals $s\in\mathcal B'$ defined by either $s(f)=f(x_0)$ for some $x_0\in E$ or by $s(f)=f(x_1)-C^{-1}f(x_2)$ for some $x_1,x_2\in E$. Then 
$$
\mathcal K_C=\left\{f\in\mathcal B: s(f)\geq 0,\,\,\forall s\in\Gamma_C\right\}.
$$
 Next, we claim that  there are constants $B,b>0$ which depend only on  the function $g$ such that for all $n\geq N_C$,  $h\in\mathcal K_C$, $s\in\Gamma_C$,  $z\in\mathbb C$ such that $|z|n\leq b$, and $u\in[0,1]$ we have 
\begin{equation}\label{comparison}
|s(Q_u^n h)-s(Q_{u,z}^n h)|\leq C^4B|z|ns(Q_u^n h).    
\end{equation}
Let us prove \eqref{comparison}. Let $h\in\mathcal K_C$.
We first consider the case when $s$ has the form $s(f)=f(x_0)$ for some $x_0\in E$. Let $Y_{j,u}$ be a stationary Markov chain with transition operator $Q_u$. Then 
for any $x\in E$, $u\in[0,1]$, and $z\in\bbC$, we have
\begin{eqnarray}\label{Relying on...}
|s(Q_{u,z}^{n}h)-s(Q_u^{n}h)|=
|Q_{u,z}^{n}h(x_0)-Q_u^{n}h(x_0)|
\\=\left|\mathbb E[h(e^{z\sum_{j=1}^ng(u,Y_{j,u})}-1)|Y_{0,u}=x_0]\right|\nonumber\\\leq
\left\|e^{z\sum_{j=1}^ng(u,Y_{j,u})}-1 \right\|_\infty
Q_u^{n}h(x_0)= \left\|e^{z\sum_{j=1}^ng(u,Y_{j,u})}-1 \right\|_\infty
s(Q_u^{n}h)\nonumber,
\end{eqnarray}
where we used that $h$ is nonnegative.
By the mean value theorem, when $|z|n\sup|g|\leq 1$, we have 
\begin{equation}\label{z V om est}
\|e^{z\sum_{j=1}^ng(u,Y_{j,u})}-1\|_\infty\leq e|z|n\sup|g|.
\end{equation}

Next, 
consider the case when $s$ has the form $s(f)=f(x_0)-C^{-1}f(x_0')$
for some $x_0,x_0'\in E$.
Since $Q_u^{n}h\in \cK_{C/2}$, 
\begin{eqnarray}\label{FIR}
s(Q_u^{n}h)=Q^{n}_uh(x_0)-C^{-1}Q_u^{n}h(x_0')
\geq
(2C^{-1}-C^{-1})Q_u^{n}h(x_0')=
C^{-1}Q_u^{n}h(x_0').
\end{eqnarray}
By the definition of $Q_{u,z}$ and \eqref{Relying on...}, we have
\begin{eqnarray*}
|s(Q_{u,z}^{n}h)-s(Q_u^{n}h)|\\=
\big|Q_{u,z}^{n}h(x_0)-	C^{-1}Q_{u,z}^{n}h(x_0')-
\big(Q_u^{n}h(x_0)-C^{-1}Q_u^{n}h(x_0')\big)\big|\\=
\big|\big(Q_{u,z}^{n}h(x_0)-Q_u^{n}h(x_0)\big)-C^{-1}
\big(Q_{u,z}^{n}h(x_0')-Q_u^{n}h(x_0')\big)\big|\\\leq
\|e^{z\sum_{j=1}^ng(u,Y_{j,u})}-1\|_\infty\big(Q_u^{n}h(x_0)+
C^{-1}Q_u^{n}h(x_0')\big)\\\leq
\|e^{z\sum_{j=1}^ng(u,Y_{j,u})}-1\|_\infty\big((C/2)Q_u^{n}h(x_0')+
C^{-1}Q_u^{n}h(x_0')\big)\\=
\|e^{z\sum_{j=1}^ng(u,Y_{j,u})}-1\|_\infty (C/2+C^{-1})Q_u^{n}h(x_0'),
\end{eqnarray*}
where in the second inequality we used that $Q_u^{n}h\in\cK_{C/2}$.
The estimate \eqref{comparison} for this type of $s$   follows now from \eqref{z V om est} and
\eqref{FIR}.

Next, 
let $\mathcal K_{C,\mathbb C}$ be the canonical complexification of $\mathcal K_C$ (see \cite[Appendix A]{HK}).  Then, using \cite[Lemmata 6.4.1, 6.4.2]{HK}, using \eqref{comparison}, and then arguing like in \cite[Section 6.7]{HK}, 
we see that the conditions of \cite[Theorems 4.2.1 and 4.2.2]{HK} hold for the operators $Q_{u,z}$ for $z$ belonging to a neighborhood of the complex origin that does not depend on $u$ (say, for $|z|\leq r_0$ for some $r_0>0$) and with constants that do not depend on $u$.
Applying \cite[Theorems 4.2.1 and 4.2.2]{HK} with the operators $Q_{u,z}$ and the cones $\mathcal K_{C,\mathbb C}$,  we conclude that for all $z$ such that $|z|\leq r_0$ and all $u\in[0,1]$ there are uniformly bounded in $u$ and analytic in  $z$ nonzero complex numbers $\lambda(u,z)$, functions $\upsilon(u,z)\in \mathcal K_{C,\mathbb C}$ and linear functionals $\varphi_{u,z}\in (\mathcal K_{C,\mathbb C})^*$ such that $\langle\varphi_{u,z}, h_{u,z}\rangle=\langle\varphi_{u,z}, \textbf{1}\rangle=1$, $\lambda(u,0)=1$, $\upsilon(u,0)=\textbf{1}$ and $\varphi_{u,0}=\pi_u$, and all the properties in the theorem hold except for the Lipschitz continuity in $u$ and the behavior of the derivatives. To see, for instance, why Condition 4 holds, note that by analyticity and uniform boundedness, and since $\lambda(u,0)=1$, there exists a constant $C_0>0$ such that $\sup_{u\in[0,1]}|\lambda(u,z)-1|\leq C|z|$. Thus, we get this condition by possibly decreasing $r_0$.
\vskip0.1cm
\textbf{Second derivatives}: 
Recall that by the analyticity in $z$ and Lemma \ref{Derivatives of twisted transfer operator lemma}, there exists $C_1>0$ such that $$\sup_{u\in[0,1]}\sup_{|z|< \delta_5}\left( \|\partial_z N_{u,z}\|_\infty + \|\partial_z^2 N_{u,z}\|_\infty     \right) \leq C_1.$$   
Now, by Lemma \ref{Spectral Decomposition Lemma} and the continuity of $N_{u,z}$ in $(u,z)$, there exists $r_1\in(0,1)$ and $m_0$ large enough such that 
\begin{align}
    \sup_{|z|< r_1} \sup_{u\in [0,1]}r(N_{u,z})\leq \|N_{u,z}^{m_0}\|_\infty^{1/m_0}<r_1. \label{Spectral radius twisted N}
\end{align}
We thus decompose each $n\in \mathbb Z^+$ as $n = qm_0 + \ell$.  Thus, it can be shown that there exists $r < r_1 < 1$ such that 
$$\sup_{|z|< r_1} \sup_{u\in [0,1]} r(N_{u,z} ) < r < r_1.$$

Now, by \eqref{Spectral radius twisted N}, and by differentiating the product $N_{u,z}^n$ twice, we obtain 
\begin{align*}
\partial_z^2 N_{u,z}^n
&= \sum_{j=0}^{n-1} N_{u,z}^j 
(\partial_z^2 N_{u,z}) N_{u,z}^{n-1-j} + 2\sum_{\substack{i,j,k\geq0\\i+j+k=n-2}} N_{u,z}^i (\partial_z N_{u,z}) N_{u,z}^j(\partial_z N_{u,z}) N_{u,z}^k.
\end{align*}
Thus,
\begin{align*}
\left\|\partial_z^2 N_{u,z}^n\right\|_\infty
&\leq \sum_{j=0}^{n-1} \|N_{u,z}^j\|_\infty \|
\partial_z^2 N_{u,z}\|_\infty \|N_{u,z}^{n-1-j}\|_\infty\\
&+ 2\sum_{\substack{i,j,k\geq0\\i+j+k=n-2}} \|N_{u,z}^i\|_\infty \| \partial_z N_{u,z}\|_\infty \|N_{u,z}^j\|_\infty \|\partial_z N_{u,z}\|_\infty \|N_{u,z}^k\|_\infty \\
&\leq C^2 C_1 \sum_{j=0}^{n-1} \rho^j\rho^{n-1-j} + 2C^3C_1^2 \sum_{\substack{i,j,k\geq0\\i+j+k=n-2}} \rho^{i+j+k}\\
&\leq C^2 C_1n\rho^{n-1} + 2C^3C_1^2n^2\rho^{n-2}.
\end{align*}
Thus, there exists $C_2 > 0$ such that
$$\sup_{u\in [0,1]} \sup_{|t| < \delta_5} \| \partial_z^2 N_{u,z}^n\|_\infty \leq C_2 n^2 \rho^{n-2}.$$

\textbf{Lipschitz estimates}: First, for each $u,v\in [0,1]$, we write \begin{align*}
    (Q_{u,z} - Q_{v,z})f(x) &= \int e^{zg(u,y)} f(y) (Q_u - Q_v)(x,dy)\\
    &+ \int (e^{zg(u,y)} - e^{zg(v,y)})f(y) Q_v(x,dy).
\end{align*}
Thus, for $|z| < \delta,$
$$\| Q_{u,z} - Q_{v,z}   \|_\infty \leq  e^{K\delta}L|u-v| + C_\delta C_g \delta |u-v|,$$
implying that there exists $C_Q > 0$ such that 
$$\sup_{|z| < \delta}\| Q_{u,z} - Q_{v,z}   \|_\infty \leq C_Q |u-v|.$$

Recall that 
\begin{align}
    \upsilon(u,z) = \frac{\Pi_{u,z} \mathbf 1}{\pi_u(\Pi_{u,z} \mathbf 1)}.\label{upsilon projection}
\end{align}
Also, the denominator is uniformly bounded away from zero for small $z$ because at $z=0$, $\pi_u(\Pi_{u,0} \mathbf 1) =  \mathbf 1$. 

Next, we establish Lipschitz continuity for $\upsilon(u,t).$ By the uniform spectral separation, there exists a simple closed contour containing only $\lambda(u,z)$, independent of $u$ and $z$, that contains only $\lambda(u,z)$ and no other eigenvalue of $\sigma(Q_{u,z})$. The spectral projection obtained from the Riesz projection is $\Pi_{u,z} = \frac{1}{2\pi i} \int_\Gamma (\zeta I-Q_{u,z})^{-1}d\zeta.$ Because $$(\zeta I-Q_{u,z})^{-1} - (\zeta I-Q_{v,z})^{-1} = (\zeta I-Q_{u,z})^{-1} (Q_{u,z}-Q_{v,z}) (\zeta I-Q_{v,z})^{-1},$$ we obtain $$\sup_{|z| < \delta_5} \|\Pi_{u,z}-\Pi_{v,z}\|_\infty \leq C_\Pi |u-v|.$$
Thus, the Lipschitz bounds for $\Pi_{u,z}$ and $\pi_u$, along with \eqref{upsilon projection} immediately give 
$$\| \upsilon(u,z) - \upsilon(v,z)\|_\infty\leq C|u-v|.$$

Recalling that $\lambda(u,z) = \langle \pi_u, Q_{u,z}\upsilon_{u,z}\rangle$, we write
$$
\lambda(u,z) -\lambda(v,z)= \langle \pi_u, Q_{u,z}(\upsilon(u,z) -\upsilon(v,z))\rangle 
$$
$$
+ \langle \pi_u, (Q_{u,z} - Q_{v,z})\upsilon(u,z) \rangle+ \langle \pi_u - \pi_v, Q_{v,z}\upsilon(u,z)\rangle. $$
Combined with the Lipschitz continuity of $
\pi_u$, $\upsilon(u,t)$, and $Q_{u,z}$, we obtain 
$$\sup_{|z|< \delta}|\lambda(u,z) - \lambda(v,z)| \leq C |u-v|.$$

Next, define $\tilde{\varphi}:= \pi_u\circ\Pi_{u,z}$, and observe that $\tilde{\varphi}$ is a scalar multiple of $\varphi$. Thus,
\begin{align*}
    \|\tilde{\varphi}(u,z)-\tilde{\varphi}(v,z)\|_{\mathfrak B'} &= \|(\pi_u-\pi_v)\circ\Pi_{u,z} +\pi_v\circ(\Pi_{u,z}-\Pi_{v,z})\|_{\mathfrak B'} \\
    &\leq \|\pi_u-\pi_v\|_{\mathfrak B'} \|\Pi_{u,z}\|_{\infty} + \|\pi_v\|_{\mathfrak B'} \|\Pi_{u,z}-\Pi_{v,z}\|_{\infty} \\
    &\leq C|u-v|.
\end{align*}
The Lipschitz bounds for $\varphi(u,z)$ follow.

Finally, for $N_{u,z} = Q_{u,z}\Pi_{H_{u,z}} = Q_{u,z}(I-\Pi_{u,z})$, we have
\begin{align*}
    \|N_{u,z} - N_{v,z}\|_\infty &= \| Q_{u,z}(I-\Pi_{u,z})  -  Q_{v,z}(I-\Pi_{v,z})\|_\infty \\
    &=\|(Q_{u,z}-Q_{v,z})(I-\Pi_{u,z}) + Q_{v,z}(\Pi_{v,z}-\Pi_{u,z})\|_\infty\\
    &\leq \|Q_{u,z} - Q_{v,z}\|_\infty \|I - \Pi_{u,z}\|_{\infty} + \| Q_{v,z} \|_\infty \|  \Pi_{u,z} - \Pi_{v,z}  \|_\infty,
\end{align*}
giving the Lipschitz estimate for $N_{u,z}$.


\end{proof}

\section{Products of Operators and estimates on characteristic functions}

In this section, we consider products of Markov operators. We use the convention that the product of the operators $\prod_{k=1}^n A_k$ means $A_1A_2\cdots A_n$ and an empty product of operators is the identity. For $p_k\in \mathbb C$, we define $\prod_{j=k}^{k-1}p_j = 1$. 

\begin{lemma}
    Make  Assumptions \ref{MainAssm} (i)-(iv). Consider a product of $M$ consecutive transition operators $Q_{k/n}Q_{(k+1)/n}\cdots Q_{(k+M-1)/n}$. Then for each fixed $M\in \mathbb Z^+$,  for sufficiently large $n\in\mathbb Z^+$, we have
$$\sup_{1\leq k \leq n-M+1} \left\| Q_{k/n}Q_{(k+1)/n}\cdots Q_{(k+M-1)/n} -\Pi_{k/n} \right\|_{\infty} <\frac 12.$$
\end{lemma}

\begin{proof}
    Fix $M, n \in \mathbb Z^+$. The continuity assumption gives
$$\|Q_{(k+\ell)/n}-Q_{k/n}\|_\infty\leq \frac{L\ell}{n}$$ for  $0\leq\ell\leq M-1$.
Thus, $$\max_{0\leq\ell\leq M-1} \|Q_{(k+\ell)/n}-Q_{k/n}\|_{\infty} =O\left(\frac{M}{n}\right).$$  

Furthermore, we have the following telescoping representation:  
\begin{align}
&Q_{k/n}Q_{(k+1)/n}\cdots Q_{(k+M-1)/n} -Q_{k/n}^M \nonumber \\
&=\sum_{\ell=0}^{M-1}Q_{k/n}^{\ell} \left(Q_{(k+\ell)/n}-Q_{k/n}\right) Q_{(k+\ell+1)/n}\cdots Q_{(k+M-1)/n}. \nonumber
\end{align}
By the submultiplicativity of the operator norm, we have 
\begin{align}
\left\| Q_{k/n}Q_{(k+1)/n}\cdots Q_{(k+M-1)/n} -Q_{k/n}^M \right\|_\infty   \leq \frac{L}{n} \frac{M(M-1)}{2}    := \frac{C_M}{n}. \label{approximation 2}
\end{align}

Observe now that by Assumption \ref{MainAssm} (ii) (which establishes a Dobrushin contraction), there exists a uniform exponential spectral gap on $\mathfrak B$ for $Q_u$. This then implies that there exist  $\rho \in (0,1)$ and $C>0$ such that for each $M\in \mathbb Z^+$ and $u\in[0,1],$ 
\begin{align}
   \|Q_u^M - \Pi_u\|_\infty \leq C\rho^M. \label{approximation 1}
\end{align} 
Thus, combining the bounds  \eqref{approximation 2} and \eqref{approximation 1}, we have shown that 
\begin{align*}
    \left\| Q_{k/n}Q_{(k+1)/n}\cdots Q_{(k+M-1)/n} -\Pi_{k/n} \right\|_{\infty} \leq C\rho^M+\frac{C_M}{n}.
\end{align*} Choose $M$ large enough so that $C\rho^M < 1/4$, and then choose $n$ large enough so that $C_M/n< 1/4$. Then,
\begin{align*}
    \sup_{1\leq k \leq n-M+1}\left\| Q_{k/n}Q_{(k+1)/n}\cdots Q_{(k+M-1)/n} -\Pi_{k/n} \right\|_{\infty}< \frac 12.
\end{align*}
\end{proof}

In proving a local limit theorem, we will consider bounds for products of the twisted operators in the cases where $t$ is close to $0$ and away from $0$ separately. Additionally, in proving Edgeworth expansions, we need to consider the case where $|t|$ has some growth in $n$.

\subsection{Estimates for $t$ near $0$}
Denote $\Lambda_n(z)=\ln\mathbb E[e^{zS_n}]$. 
\begin{lemma}\label{GrowthLemma}
Make  Assumptions \ref{MainAssm} (i)-(iv).
For every $m\in\mathbb Z^+$, there exist $\delta_m,c_m>0$ such that for all $n$,
$$
\sup_{t\in[-\delta_m,\delta_m]}|\Lambda_n^{(m)}(it)|\leq c_mn.
$$
\end{lemma}
\begin{proof}
We have 
$$
E[e^{it S_n}]=\mathbb E\left[\left(\prod_{k=1}^{n}Q_{k/n,it}\right)\textbf{1}\right].
$$
Now, arguing like in the proof of \cite[Proposition 3.1]{koralov2025local} we have
\begin{equation}\label{Recall}
\left(\prod_{k=1}^{n}Q_{k/n,z}\right)\textbf{1}=\prod_{k=1}^{n}\lambda(k/n,z)\prod_{k=1}^{n-1}\langle\varphi(k/n,z),\upsilon((k+1)/n,z)\rangle\cdot \upsilon(z,0)    
\end{equation}
$$
+\prod_{k=1}^{n}\lambda(k/n,z)U_n(z),
$$
where $\|U_n(z)\|_\infty\leq C/n$. Now, as $\langle\varphi(k/n,z),\upsilon(k/n,z)\rangle=1$ we have 
$$
\langle\varphi(k/n,z),\upsilon((k+1)/n,z)\rangle=1+O(1/n).
$$
Moreover, as $\upsilon(0,0)=1$ we have $\upsilon(z,0)=1+O(|z|)$. Furthermore, as $\lambda(k/n,0)=1$, we have
$$
\lambda(k/n,z)=1+O(|z|).
$$
Therefore, taking the logarithm, we see that there exist $c,\delta,N_0>0$ such that if $|z|<\delta$ and $n\geq N_0$, then
$$
|\Lambda_n(z)|\leq cn.
$$
Take $m\in\mathbb N$.
Using the Cauchy integral formula, we conclude that, on possibly a smaller neighborhood of the origin, we have 
$$
|\Lambda_n^{(m)}(z)|\leq c_mn
$$
for some constant $c_m>0$.
\end{proof}

\subsection{Estimates for $t$ away from $0$}
\begin{lemma}\label{t cpt}
Make  Assumptions \ref{MainAssm} (i)-(v).
For every compact set $K\subset\mathbb R\setminus\{0\}$, there are constants $C_K,c_k>0$ such that for all $n$,
$$
\sup_{t\in K}\left\|\prod_{k=1}^{n}Q_{k/n,it}\right\|_\infty\leq C_Ke^{-c_K n}.
$$
\end{lemma}
\begin{proof}
Arguing like in the proof of \cite[Corollary III.13]{HH}, 
using the upper semi-continuity of the spectral radius and the continuity of $(u,t)\to Q_{u,it}$, we see that there are $A_K,a_K>0$ such that for all $n$,
$$
\sup_{u\in[0,1]}\sup_{t\in K}\|Q_{u,it}^n\|_\infty\leq A_Ke^{-a_K n}.
$$
Now, let us take $n_K$ such that $A_Ke^{-a_K n}<1/4$. Let us take $\delta_K>0$ such that for all $0\leq u_0,u_1, \cdots,u_{n_K}\leq 1$ such that $|u_i-u_0|\leq \delta_K$, we have 
$$
\sup_{t\in K}\left\|Q_{u_0,it}^{n_K}-Q_{u_1,it}\circ Q_{u_2,it}\circ\cdots Q_{u_{n_K}, it}\right\|_{\infty}<\frac14.
$$
Then
$$
\sup_{t\in K}\left\|Q_{u_1,it}\circ Q_{u_2,it}\circ\cdots Q_{u_{n_K},it}\right\|_\infty<\frac12.
$$
We thus conclude that for all $k$ and $n$ such that $k+n_K<n$, when $n$ is large enough we have 
$$
\sup_{t\in K}\left\|Q_{k/n,it}\circ Q_{(k+1)/n,it}\circ\cdots \circ Q_{(k+n_K)/n,it}\right\|_\infty<\frac12.
$$
Now the result follows from the submultiplicativity of operator norms.
\end{proof}
\subsection{General estimates for uniformly elliptic Markov chains with applications for growing domains of $t$}
In this section we will mostly recall some known results under additional ellipticity conditions on the operators $Q_u$. We begin with the following result for (not necessarily uniform) lower $\psi$-mixing operators $Q_u$.
\begin{lemma}[Proposition 10 in \cite{PelLLT}]\label{PelLemma}
 Suppose that there is $a_n\in(0,1]$ such that for all $k\leq n$,
 $$
Q_{k/n}(x,\Gamma)\geq a_n\mathbb P(X_{k,n}\in\Gamma)
 $$
 for all $x\in E$ and all measurable sets $\Gamma\in\mathcal B(E)$. Then, for all real $t$,
 $$
\left|\mathbb E[e^{itS_n}]\right|\leq\exp\left(-\frac{a_n^4}{16}\sum_{k=1}^{n}\left(1-\left|\mathbb E[e^{it g(k/n,X_{k,n})}]\right|^2\right)\right). 
 $$
\end{lemma}
\begin{corollary}\label{PelCor}
Make  Assumptions \ref{MainAssm} (i)-(iv).
 Then there exists a constant $C>0$ such that for all real $t$, we have
$$
\left|\mathbb E[e^{itS_n}]\right|\leq C\exp\left(-\frac{a_n^4}{16}\sum_{k=1}^{n}\left(1-\left|\int e^{it g(k/n,x)}d\pi_{k/n}(x)\right|^2\right)\right).
$$
\end{corollary}
\begin{proof}
By Theorem  \ref{Tru}, we have 
$$
\|\pi^{(n)}_{k,1}-\pi_{k/n}\|_{\text{TV}}\leq C_1/n.
$$
Thus 
$$
\sum_{k=1}^{n}\left(1-\left|\mathbb E[e^{it g(k/n,X_{k,n})}]\right|^2\right)\geq 
\sum_{k=1}^{n}\left(1-\left|\int e^{it g(k/n,x)}d\pi_{k/n}(x)\right|^2\right)-C_1,
$$
and the conclusion follows.
\end{proof}

The next result concerns the case when $E$ is smooth, each $Q_u$ has strictly positive bounded transition densities, and $g(u,x)$ is sufficiently regular in $x$. 
\begin{lemma}[Proposition 40 in \cite{dolgopyat2023berry}]\label{DH Lemma}
Suppose that the state space $E$ of the chain is a compact  Riemannian manifold and that 
$$
Q_u(x,\Gamma)=\int_{\Gamma}p_u(x,y)dy,
$$
where $dy$ denotes the normalized volume measure and $p_u$ are uniformly bounded and bounded away from $0$. Assume also that the functions $g_u(x)=g(u,x)$ are uniformly H\"older continuous with some exponent $\alpha\in(0,1]$. There exist $\delta,c_1,C_1>0$ such that for all $t\in\mathbb R$ with $|t|\geq \delta$, we have
$$
|\mathbb E[e^{itS_n}]|\leq C_1e^{-c_1 n|t|^{1-\frac1\alpha}}.
$$
In particular, for all $d<\frac{1+\alpha}{1-\alpha}$ we have
$$
\int_{\delta\leq|t|\leq n^{(d-1)/2}} \left|\mathbb E[e^{itS_n}]/t\right|dt=o(n^{-d/2}).
$$
\end{lemma}


\section{Proof of the limit theorems}
\subsection{Proof of Theorem \ref{Var}} 
Recall \eqref{Recall}. Namely, for $z$ small enough we have
$$
\left(\prod_{k=1}^{n}Q_{k/n,z}\right)\textbf{1}=\prod_{k=1}^{n}\lambda(k/n,z)\prod_{k=1}^{n-1}\langle\varphi(k/n,z),\upsilon((k+1)/n,z)\rangle\cdot \upsilon(0,z)
$$
$$
+\prod_{k=1}^{n}\lambda(k/n,z)U_n(z),
$$
where $\|U_n(z)\|_\infty\leq C/n$. 
Thus, 
\begin{equation}\label{ii}
 \Lambda_n(z)=O(1)+O(1/n)+\sum_{k=1}^{n}\ln\lambda(k/n,z)+
\sum_{k=1}^{n-1}\ln \langle\varphi(k/n,z),\upsilon((k+1)/n,z)\rangle.   
\end{equation}
Now, by the Cauchy integral formula, we see that 
$$
\text{Var}(S_n)=\Lambda_n''(0)=O(1)+O(1/n)+\sum_{k=1}^{n}(\ln\lambda(k/n,z))''|_{z=0}
$$
$$
+\sum_{k=1}^{n-1}\left(\ln \langle\varphi(k/n,z),\upsilon((k+1)/n,z)\rangle\right)''|_{z=0}.
$$
Since the above sums are Riemann sums of functions with a uniform Lipschitz constant, we thus get that
$$
\lim_{n\to\infty}\frac1n\text{Var}(S_n)=\int_{0}^1(\ln\lambda(u,z))''|_{z=0}\,du+\int_{0}^1
\left(\ln \langle\varphi(u,z),\upsilon(u,z)\rangle\right)''|_{z=0}\,du
:=a^2. 
$$
Notice that 
$$
\ln \left\langle\varphi(u,z),\upsilon(u,z)\right\rangle=0,
$$
and so the second integral on the above right-hand side vanishes.
Now, the functions $u\to\lambda(u,z),\varphi(u,z),\upsilon(u,z)$ are Lipschitz continuous, uniformly in $z$. We thus see that 
$$
\sup_n\left|\frac1n\text{Var}(S_n)-a^2\right|<\infty.
$$
Now, as the operators $Q_{u,it}$ have spectral radius smaller than one for $t\not=0$ (small enough), we must have $a>0$. Indeed, the second derivative of $\ln\lambda(u,z)$ at $z=0$ is  
$$
\lim_{n\to\infty}\frac1n\text{Var}\left(\sum_{j=1}^n g(u,Y_{j,u})\right),
$$
where $Y_{j,u}$ is the stationary Markov chain with transition operator $Q_u$,
and so it is positive as $r(Q_{u,it})<1$ for sufficiently small $t$.

Let us prove  the existence of $b\in\mathbb R$ such that 
$$
\sup_n|\mathbb E[S_n]-bn|<\infty.
$$
By \eqref{ii} and the Cauchy integral formula, we see that 
$$
\mathbb E[S_n]=\Lambda_n'(0)=O(1)+O(1/n)+\sum_{k=1}^{n}\left(\ln\lambda(k/n,z)\right)'|_{z=0}
$$
$$
+\sum_{k=1}^{n-1}\left(\ln \langle\varphi(k/n,z),\upsilon((k+1)/n,z)\rangle\right)'|_{z=0}.
$$
Since the above sums are Riemann sums of functions with a uniform Lipschitz constant and 
$\ln \langle\varphi(u,z),\upsilon(u,z)\rangle=0,$ 
we get the result with 
$$
b=\int_{0}^1(\ln\lambda(u,z))'|_{z=0}\,du.
$$
Finally, note that 
$$
(\ln\lambda(u,z))'|_{z=0}=\int g(u,x)d\pi_u(x).
$$

To see why we also have 
$$
\sup_n\left|\mathbb E[(S_n-bn)^2]-na^2\right|<\infty.
$$
We note that 
$$
\mathbb (E[S_n])^2-n^2b^2=(\mathbb E[S_n]+bn)(\mathbb E[S_n]-bn)=O(n)O(1/n)=O(1).
$$
\qed

\subsubsection{Higher order cumulants and moments}\label{Rem1}
Let $k\geq 3$ be an integer.  Let $\Gamma_k(S)$ be the $k$-th cumulant of a bounded random variable $S$, which we recall is the $k$-th derivative at $z=0$ of $\ln\mathbb E[e^{zS}]$.
Set 
$$
\gamma_k=\int_{0}^1(\ln\lambda(u,z))^{(k)}|_{z=0}\,du.
$$
Then $b=\gamma_1$ and $a^2=\gamma_2$.
Note that 
$$
(\ln\lambda(u,z))^{(k)}|_{z=0}=\lim_{n\to\infty}\frac1n\Gamma_k(S_{n,u}),
$$
where 
$$
S_{n,u}=\sum_{j=1}^ng(u,Y_{j,u})
$$
and  $(Y_{j,u})_{j=1}^\infty$ is the stationary Markov chain with transition operator $Q_u$.
\begin{theorem}\label{HighThm}
For all $k\geq3$, we have 
$$
\sup_n|\Gamma_k(S_n)-n\gamma_k|<\infty.
$$    
Moreover, there exist $C_k>0$ independent of $Q_u,\pi_u$, such that for all even $k\geq 3$,
$$
\left|\mathbb E[(S_n-bn)^k]-C_kn^{k/2}a^{k}\right|=O(n^{k/2-1}),
$$
while for all odd $k\geq 3$,
$$
\left|\mathbb E[(S_n-bn)^k]-C_kn^{[k/2]}a^{k-3}\gamma_3\right|=O(n^{[k/2]-1}).
$$
\end{theorem}
\begin{proof}
By differentiating $k$ times, it follows that   
$$
\sup_n|\Gamma_k(S_n)-n\gamma_k|<\infty.
$$    
The asymptotic formula for the moments follows from the general relation between moments and cumulants and the above estimate. Indeed, suppose $b=0$. Denote $M_n(z)=\mathbb E[e^{zS_n}]$. Then $M_n(z)=e^{\Lambda_n(z)}$, and so
by the Fa\'a di Bruno formula \cite[Section 1.3]{FDB}, we have
$$
\mathbb E[(S_n)^k]=(M_n)^{(k)}(0)=\sum_{m_1,\cdots,m_k}\frac{k!}{\prod_{j=1}^k(m_j)!}\prod_{j=1}^k\left(\frac{\Gamma_j(S_n)}{j!}\right)^{m_j}
$$
where $m_1,\cdots,m_k$ are nonnegative integers satisfying $\sum_{j=1}^kjm_j=k$. Thus,
$$
\mathbb E[(S_n)^k]=\sum_{m_1, \cdots,m_k}n^{m_1+ \cdots +m_k}\frac{k!}{\prod_{j=1}^k(m_j)!}\prod_{j=1}^k\left(\frac{\Gamma_j(S_n)/n}{j!}\right)^{m_j}.
$$
Note that as $b=0$ we have $\Gamma_1(S_n)/n=O(1/n)$ and so we can disregard these terms, namely assume that $m_1=0$. In this case, $\sum_{j=2}^{k}m_j\leq 2\sum_{j=2}^k jm_j=k$ and so $\sum_{j=2}^{k}m_j\leq k/2$. Thus, as $\Gamma_j(S_n)/n$ are bounded in $n$, we see that we can disregard all the terms for which $\sum_{j=2}^{k}m_j<k/2$. Note that for even $k$, we are left with the tuple $(m_1,\cdots,m_k)=(0,k/2,0,\cdots,0)$, and so 
$$
\mathbb E[(S_n)^k]=C_kn^{k/2}(\Gamma_2(S_n)/n)^{k/2}+O(n^{k/2-1}).
$$
Using that $\gamma_2(S_n)/n=a^2+O(1/n)$, we can  replace  $\gamma_2(S_n)/n$ with $a^2$. 

For odd $k$, we are left with the tuple $(m_1,\cdots,m_k)=(0,(k-3)/2,1,0,\cdots,0)$, which gives us
$$
\mathbb E[(S_n)^k]=C_kn^{[k/2]}(\Gamma_2(S_n)/n)^{(k-3)/2}\Gamma_3(S_n)/n+O(n^{[k/2]-1}),
$$
and again we can replace $\Gamma_2(S_n)/n$ by $a^2$ and $\Gamma_3(S_n)/n$ by $\gamma_3$.
\end{proof}
\begin{remark}
Instead of disregarding the term $O(n^{[k/2]-1})$, we can actually replace it with a correction term up to an $O(1)$ error term. Thus, the proof above actually shows that we can expand the moments of $n^{-[k/2]}(S_n-bn)$ in negative powers of $n$ up to an error term $O(n^{-[k/2]})$. 
\end{remark}

\subsection{Proof of Theorem \ref{BE}}
The theorem follows from Lemma \ref{GrowthLemma}, \cite[Proposition 24]{dolgopyat2023berry}, Theorem \ref{Var}, and the results in \cite[Section 3.2]{hafouta2022non} (which allow us to pass from the standardized case to $(S_n-bn)/\sqrt n$). 
\qed

\subsection{Proof of Theorem \ref{LLT}}
First, since $a>0$ by Lemma \ref{GrowthLemma} and \cite[Corollary 28]{dolgopyat2023berry},
we see that there exist $c,C,\delta_0>0$  such that for all $n$ and $t\in[-\delta_0,\delta_0]$, we have 
\begin{equation}\label{ff}
|\mathbb E[e^{itS_n}]|\leq Ce^{-ct^2n}    
\end{equation}
 Thus, Theorem \ref{LLT} follows from Theorem \ref{BE}, \cite[Theorem 2.2.3]{HK}, Lemma \ref{t cpt} and \eqref{ff}.

\qed
\subsection{Proof of Theorem \ref{Edge0}}
The theorem follows from Lemma \ref{GrowthLemma} and Lemma \ref{t cpt} together with \cite[Proposition 25]{dolgopyat2023berry} and the results in \cite[Section 5.1]{hafouta2026non} (which allow us to pass from the standardized case to $(S_n-bn)/\sqrt n$).

\qed
\subsection{Proof of Theorems \ref{Edge1} and \ref{Edge2}}
Theorems \ref{Edge1} and \ref{Edge2} follow from Lemma \ref{GrowthLemma}, Lemma \ref{t cpt}, Corollary \ref{PelCor}, and Lemma \ref{DH Lemma} together with \cite[Proposition 25]{dolgopyat2023berry} and the results in \cite[Section 5.1]{hafouta2026non} (which again allow us to pass from the standardized case to $(S_n-bn)/\sqrt n$).

\subsection{Proof of Theorem \ref{LDP}}
Suppose $b=0$ (otherwise we can replace $g$ with $g-b$). Using \ref{Recall} with $z=t\in \mathbb R$ with $|t|$ small enough (say, smaller than some $\epsilon_0$), we see that 
$$
\lim_{n\to\infty}\frac1n\Lambda_n(t)=\int_{0}^1 \ln\lambda(u,t)du.
$$
Using \cite[Lemma XIII.2]{HH}, we get the result with the Fenchel--Legendre transform $c(\cdot)$ of $t\to\int_{0}^1 \ln\lambda(u,t)du$, given by
$$
c(\epsilon)=\sup_{|t|\leq\epsilon_0}\left(t\epsilon-\int_{0}^1 \ln\lambda(u,t)du\right).
$$
We note that $c(\cdot)$ has the desired properties since, by assumption, $a>0$, the function $t\to \int_{0}^1 \ln\lambda(u,t)du$ is analytic in $t$, and
$$
a^2=\frac{d^2}{dt^2}\left(\int_{0}^1 \ln\lambda(u,t)\,du\right)\Big|_{t=0}=\int_{0}^1 (\ln\lambda(u,t))''|_{t=0}\,du.
$$
\qed

\subsection{Proof of Theorem \ref{MDP}}
Suppose again $b=0$. Let $(c_n)_{n=1}^\infty$ be a sequence satisfying  $\lim_{n\to\infty}\frac{c_n}{\sqrt n}=\infty$.
Let $t\in\mathbb R$. Recall that $\Lambda_n(0)=0$, $\Lambda_n'(0)=\mathbb E[S_n]=O(1/n)$, $\Lambda_n''(0)=na^2+O(1/n)$. Moreover, recall that 
$$
a^2=\int_{0}^1(\ln\lambda(u,z))''|_{z=0}\,du
$$
and 
$$
b=\int_{0}^1(\ln\lambda(u,z))'|_{z=0}\,du.
$$
Next, by  \eqref{Recall} with $z=\xi/d_n$, $d_n=n/c_n$,
we have
$\sup_{|\xi|\leq |t|}|\Lambda'''(\xi/d_n)|=O(n)$. Thus, by the Lagrange form of the second-order Taylor remainder, 
for  $n$ large enough, we have
$$
\Lambda_n(t/d_n)=O(1/n)+\frac{t^2(na^2+O(1/n))}{2d_n^2}+O(nd_n^{-3}).
$$
Consequently,
$$
\lim_{n\to\infty}\frac1{c_n^2/n}\Lambda_n(t/d_n)=\frac12a^2t^2.
$$
Now the result follows from the G\"artner--Ellis theorem. 
\qed

\newpage
\printbibliography

\end{document}